\documentclass[12pt]{article}
\usepackage{amsmath,amsfonts,amsthm,amscd,upref,amstext}
\usepackage[dvips]{graphicx}
\usepackage{subfigure}
\usepackage{pb-diagram,pb-xy}
\usepackage[all]{xy}
\usepackage[dvips]{color}

\newtheorem{prop}{Proposition}[section]
\newtheorem{thm}[prop]{Theorem}
\newtheorem{cor}[prop]{Corollary}

\newtheorem{defn}[prop]{Definition}
\newtheorem{rem}[prop]{Remark}
\newtheorem{lem}[prop]{Lemma}

\numberwithin{equation}{section}

\begin{document}

\title{ Some Results about Endomorphism Rings for Local Cohomology Defined by a Pair of Ideals}
\author{V. H. Jorge P\'erez$^{1,}\,$\thanks{Work
partially supported by CNPq-Brazil - Grants 309316/2011-1,
and FAPESP Grant 2012/20304-1 }\,\,\,\, \,\,\, T.H. Freitas $^{2,}$
\thanks{Work partially supported by FAPESP-Brazil - Grant
08915/2006-2.  2000 Mathematics Subject
Classification: 13D45. {\it Key words}: Local Cohomology, Matlis Duality, Endomorphism ring. }}

\date{}
\maketitle

\vspace{0.3cm}
$^1$ Universidade de S{\~a}o Paulo -
ICMC, Caixa Postal 668, 13560-970, S{\~a}o Carlos-SP, Brazil ({\it
e-mail: vhjperez@icmc.usp.br}).

\vspace{0.3cm}
 $^2$ Universidade de S{\~a}o Paulo -
ICMC, Caixa Postal 668, 13560-970, S{\~a}o Carlos-SP, Brazil ({\it
e-mail: tfreitas@icmc.usp.br}).

\vspace{0.3cm}
\begin{abstract}
Let $(R,\mathfrak{m},k)$ denote a local ring. For $I$ and $J$ ideals of $R$, for all integer $i$,  let $H^i_{I,J}(-)$ denote the $i$-th local cohomology functor with respect to $(I,J)$. Here we give a generalized version of Local Duality Theorem for local cohomology defined by a pair of ideals. Also, for $M$ be a finitely generated  $R$-module,  we study the behavior of the endomorphism rings $H^t_{I,J}(M)$ and $D(H^t_{I,J}(M))$ where $t$ is the smallest integer such that the local cohomology with respect to a pair of ideals is non-zero and $D(-):= {\rm Hom}_R(-,E_R(k))$ is the Matlis dual functor. We show too that if $R$ be a  $d$-dimensional complete Cohen-Macaulay and  $H^i_{I,J}(R)=0$ for all $i\neq t$,
the natural homomorphism
$R\rightarrow {\rm Hom}_R(H^t_{I,J}(K_R), H^t_{I,J}(K_R))$ is an isomorphism and for all $i\neq t$, where $K_R$  denote the canonical module of $R$.

\end{abstract}

\maketitle
\section{Introduction}
\hspace{0.5cm}Throughout this article $R$ is a commutative Noetherian (non-zero identity) ring and $\mathfrak{a},I,J$ be ideals of $R$ and $M$ be a non-zero $R$-module. For $i \in \mathbb{N}$, $H^i_{\mathfrak{a}}(M)$ denote the $i$-th  local cohomology module of $M$ with respect to $\mathfrak{a}$ (see \cite{grot}, \cite{Hart}, \cite{24h}). This concept has been an important tool in algebraic geometry and commutative algebra and has been studied by several authors.

For an $R$-module $M$ consider the natural homomorphism $R\rightarrow {\rm Hom}_R(M,M)$, $r \mapsto f_r$ where $f_r(m):= rm$, for all $m\in M$ and $r\in R$. This map in general is neither injective nor surjective.

Let $D(-):= {\rm Hom}_R(-,E_R(k))$ is the Matlis dual functor. For an $R$-module $M$, the module $D(M)$ admits a structure of a $\widehat{R}$-module. The natural injective homomorphism $M\rightarrow D(D(M))$ induces an injective homomorphism
$${\rm Hom}_R(M,M)\rightarrow {\rm Hom}_{\widehat{R}}(D(M),D(M))$$ such that the diagram  below is commutative
$$
\begin{array}{lcr}
 \begin{array}{c}
 \\
 \\
 \\
 \\
 \end{array}&\begin{array}{ccccc}
  R& \rightarrow & {\rm Hom}_R(M,M) & &  \\
  \downarrow &  & \downarrow   &   & \\
  \widehat{R} & \rightarrow & {\rm Hom}_{\widehat{R}}(D(M),D(M)).& &
   \end{array}
   & \begin{array}{c}
 \\
 \\
 \\
 \\
 \end{array}
\end{array}
$$

For the ordinary local cohomology, the study of endomorphism rings ${\rm Hom}_R(H^i_{I}(M), H^i_{I}(M))$ was initially discussed in \cite{hochhuneke} for the case ${\rm dim}R=i$ and $I=\mathfrak{m}$. For certain ideals $I$ and several $i \in \mathbb{N}$, there a lot of works on the study of this structure for local cohomology. We cite \cite{ScheMatlis}, \cite{HellusStuck}, \cite{ScheEnd}, \cite{ScheStruct}, \cite{Kazem} and \cite{waqas} for more results in this sense.

In \cite{art1}, Takahashi, Yoshino and Yoshizawa  introduced a generalization of the notion of local cohomology module, called of local cohomology defined by a pair of ideals $(I,J)$. More precisely, let $W(I,J)=\{\mathfrak{p} \in \rm{ Spec}R \mid I^n\subseteq \mathfrak{p}+J \ \ \mbox{for some integer n}  \}$ and $\widetilde{W}(I,J)$ denotes the set of ideals $\mathfrak{a}$ of $R$ such that $I^n \subseteq \mathfrak{a}+J$. Let the set of elements of $M$
$$\Gamma_{I,J}(M)=\{x\in M \ | \  I^n x\subseteq Jx \ \ for  \ \ n\gg 1\}.$$
The functor $\Gamma_{I,J}(-)$ is a left exact functor, additive and covariant, from the category of all $R$-modules, called $(I,J)$-torsion functor. For an integer $i$, the $i$-th right derived functor of $\Gamma_{I,J}(-)$ is denoted by $H^i_{I,J}(-)$ and will be call to as $i$-th local cohomology functor with respect to $(I,J)$. For an $R$-module $M$, $H^i_{I,J}(M)$ refer as the $i$-th local cohomology module of $M$ with respect to $(I,J)$ and $\Gamma_{I,J}(M)$ as the $(I,J)$- torsion part of $M$. When $J=0$ or $J$ is a nilpotent ideal, $H^i_{I,J}(-)$ coincides with the ordinary local cohomology functor $H^i_{I}(-)$ with the support in the closed subset $V(I)$.

In \cite{art1} the authors also introduce a generalization of $\check{C}$ech complexes as follows. For an element $x\in R$, let $S_{x,J}$ the set multiplicatively closed subset of $R$ consisting of all elements of the form $x^n+j$ where $j\in J$ and $n\in \mathbb{N}$. For an $R$-module $M$, let $M_{x,J}= S_{x,J}^{-1}M$. The complex $\check{C}_{x,J}$ is defined as
$$\check{C}_{ x,J}:  \ 0\rightarrow R \rightarrow R_{x,J}\rightarrow 0$$ where $R$ is sitting in the 0th position and $R_{x,J}$ in the 1st position in the complex. For a system of elements of $R$ $\underline{x}= x_1,\ldots,x_s$, let a complex $\check{C}_{\underline{x},J}= \bigotimes_{i=1}^s \check{C}_{ x_i,J} $. If $J=0$ this definition coincides with the usual $\check{C}$ech complex with respect to $\underline{x}= x_1,\ldots,x_s$.

Questions involving vanishing, artinianness, finiteness has been studied by several authors such as \cite{nonart}, \cite{artop}, \cite{art8}, \cite{art3}, \cite{payrovi}, \cite{payrovi2}, \cite{filterdepth}, \cite{amoli} and others. This concept is used also by \cite{paper} and \cite{paper2}.

Results on the behavior of endomorphism rings for local cohomology defined by a pair of ideals are not known. In this sense, the main purpose of this paper is is to give some contributions in this aspect. The organization of the paper is as follows.

In the section two, firslty we consider de least integer $i$ such that the local cohomology defined by a pair of ideals is non zero. This number is denoted by ${\rm depth}(I,J,M)$. For more information about this new concept, the reader can consult \cite{art8}.
The main result of this section,(Theorem \ref{glduality}) and one of the most important of this work, is the generalized version of the Local Duality Theorem.
This result is thought of a generalization of the local duality theorem for local rings (and non necessarily an Cohen-Macaulay ring) and generalizes too \cite[Theorem 5.1]{art1}, \cite[Lemma 2.4]{waqas} and \cite[Theorem 6.4]{paper}.

In the third section we will investigate the previous diagram in the case of local cohomology module $H^t_{I,J}(M)$. We give several sufficient conditions for the homomorphism $$R\rightarrow  {\rm Hom}_R(D(H^t_{I,J}(M), D(H^t_{I,J}(M)))$$ is an isomorphism.

In the last section, we define the truncation complex using the concept of local cohomology defined by a pair of ideals. This concept will be useful to show that if $R$ be a  $d$-dimensional complete Cohen- Macaulay and  $H^i_{I,J}(R)=0$ for all $i\neq t$,
the natural homomorphism
$R\rightarrow {\rm Hom}_R(H^t_{I,J}(K_R), H^t_{I,J}(K_R))$ is an isomorphism and for all $i\neq t$, where $K_R$  denote the canonical module of $R$.

\section{The Generalized Local Duality Theorem}

\hspace{0.5cm} Let $(R,\mathfrak{m})$ be a commutative Noetherian local ring with ideal maximal $\mathfrak{m}$ and the residue field $k= R/\mathfrak{m}$. We denote by $D(-):= {\rm Hom}_R(-,E)$  the Matlis dual fuctor, where where $E:= E_R(k)$ is the injective hull of $k$.

We knows that by \cite[Theorem 6.2.7]{grot}, for an ideal $\mathfrak{a}$ of $R$ (non necessary local ring) and a finite $R$-module $M$ with $\mathfrak{a}M\neq M$, the ${\rm depth}(\mathfrak{a},M)$ is the least integer $i$ such that $H^i_\mathfrak{a}(M)\neq 0$. With this we can consider the following definition.

\begin{defn} (\cite[Definition 3.1]{art8}) Let $I,J$ be two ideals of $R$ and let $M$ be an $R$-module. We define the depth of $(I,J)$ on $M$ by
$$\rm{depth}(I,J,M)={\rm inf}\{i \in \mathbb{N}_0 \mid H^i_{I,J}(M)\neq 0\}$$
if this infimum exists, and $\infty$ otherwise.
\end{defn}

%If we consider $M$ is a finitely generated module over a local ring $(R,\mathfrak{m})$ and

If $J\neq R$, by \cite[Theorem 4.3]{art1} and definition above we have $H^i_{I,J}(M)\neq 0$ for all
$$ {\rm depth}(I,J,M)\leq i \leq \dim M/JM.$$

In \cite[Theorem 4.1]{art1}, for any finitely generated $R$-module, the authors shows that
$${\rm depth}(I,J,M)= {\rm inf}\{{\rm depth}(M_\mathfrak{p}) \mid \mathfrak{p}\in W(I,J)\}.$$

\begin{lem}\label{auxiliar}(\cite[Proposition 3.3]{art8}) For any finitely generated $R$-module $M$ we have the equality
$${\rm depth}(I,J,N)= {\rm inf}\{ {\rm depth}(\mathfrak{a},N) \mid \mathfrak{a}\in \widetilde{W}(I,J)\}.$$
\end{lem}
\begin{proof} Denote $t= {\rm depth}(I,J,M)$ and $s= {\rm inf}\{ {\rm depth}(\mathfrak{a},M) \mid \widetilde{W}(I,J)  \}.$ So exist $\mathfrak{b}\in \widetilde{W}(I,J)$ such that ${\rm depth}(\mathfrak{b},M)= s$. By \cite[Theorem 4.1]{art1}  and the fact that $V(\mathfrak{b})\subseteq W(I,J)$ follows that
$$
\begin{array}{lll}
s={\rm depth}(\mathfrak{b},M) &= {\rm inf}\{{\rm depth}(M_{\mathfrak{p}}) \mid \mathfrak{p}\in V(\mathfrak{b})\}\\
&\geq {\rm inf}\{{\rm depth}(M_{\mathfrak{p}}) \mid \mathfrak{p}\in W(I,J)\} =t.
\end{array}
$$

If $t<s$ we have that $H^t_{\mathfrak{a}}(M)=0$ for all $\mathfrak{a}\in \widetilde{W}(I,J)$. Therefore, by \cite[Theorem 3.2]{art1}, $H^t_{I,J}(M)=0$ that is a contradiction and the proof is completed.
\end{proof}

Before the next result remember that for
$(R,\mathfrak{m}, k)$ be a $d$-dimensional Cohen-Macaulay local ring with a canonical module $K_R$
 it is well known the existence of isomorphisms
$$H^i_\mathfrak{m}(M)= {\rm \rm Ext}_R ^{d-i}(M, K_R)^\vee$$
for $0\leq i\leq d$,
where $(-)^\vee= {\rm Hom}_R(-,E_R(\mathbb{K}))$ and $H^d_{\mathfrak{m}}(R)\cong K_R^\vee$.
This result is called Local Duality Theorem. There is a generalization of this result in \cite[Theorem 5.1]{art1} and \cite[Theorem 6.4]{paper}.

Under this comments, we prove now a
generalization of the Local Duality Theorem, which
extends \cite[Theorem 5.1]{art1}, \cite[Lemma 2.4]{waqas} and \cite[Theorem 6.4]{paper}.

\begin{thm}[Generalized Local Duality]\label{glduality} Let $(R,\mathfrak{m})$  be a local ring and $I,J$ ideals of $R$. Assume that $H^i_{I,J}(R)=0$ for all $i>n$. Then for any $R$-module $M$ and $i\in \mathbb{Z}$ follows the isomorphism:
\begin{enumerate}
\item[(a)] ${\rm Tor}_{n-i}^R(M, H^n_{I,J}(R))\cong H^i_{I,J}(M)$.

\item[(b)] $D(H^i_{I,J}(M))\cong {\rm Ext}^{n-i}_R(M, D(H^n_{I,J}(R))).$
\end{enumerate}

\end{thm}
\begin{proof}If $(a)$ is true, follows by Lemma \ref{lemma2.2waqas} that
$$D(H^i_{I,J}(M))\cong D({\rm Tor}_{n-i}^R(M, H^n_{I,J}(R)))\cong {\rm Tor}_{n-i}^R(M, H^n_{I,J}(R)),$$ and we obtain $(b)$.

So, is sufficient to prove $(a)$. Consider the families of functors $\{H^i_{I,J}(-) \ \mid \  i\geq 0  \}$ and $\{ {\rm Tor}_{n-i}^R(-, H^n_{I,J}(R)) \  \mid \ i\geq 0 \}$. We want to show the isomorphism of these two families.
First note that if $i=n$ follows that
$${\rm Tor}_{0}^R(N, H^n_{I,J}(R))\cong  H^n_{I,J}(N)$$ because \cite[Lemma 4.8]{art1} shows that
$ H^n_{I,J}(N)\cong H^n_{I,J}(R)\otimes_R N$ for any $R$-module $N$ when $H^n_{I,J}(R)=0$ for all $i>n$.
Since both os families of functors induces a long exact sequence from the short exact sequence, is enough to prove that for $M$ an projective $R$-module, $$H^i_{I,J}(M)=0= {\rm Tor}_{n-i}^R(M, H^n_{I,J}(R))$$ for all $i>n$.

Firstly suppose that $M=R$ the statement is true. Since any projective  $R$-module over local ring is isomorphic to direct sum of $R$ and the functors ${\rm Tor}_{n-i}^R(-, H^n_{I,J}(R))$ and $H^i_{I,J}(-)$ commutes with direct sums, we conclude the claim.
\end{proof}

\begin{rem} If $n= {\rm depth}(I,J,M)$, this result generalizes \cite[Lemma 2.4]{waqas}. Furthermore this theorem generalizes \cite[Theorem 5.1]{art1}.
\end{rem}

The next result show the relation between the $J$-adic completion of $H^n_{\mathfrak{m},J}(R)$ and the dual of certain modules. This result generalizes \cite[Theorem 5.4]{art1}.

\begin{cor}Let $(R,\mathfrak{m})$  be a local ring and $J$ be an ideal of $R$. Assume that $H^i_{\mathfrak{m},J}(R)=0$ for all $i>n$. Then there is a natural isomorphism
$$\varprojlim \frac{H^n_{\mathfrak{m},J}(R)}{J^sH^n_{\mathfrak{m},J}(R)}  \cong D(\Gamma_J(D(H^n_{\mathfrak{m}}(R))))$$
In particular, if  $R$ be a $n$-dimensional Cohen Macaulay complete local ring we obtain that
$$\varprojlim \frac{H^n_{\mathfrak{m},J}(R)}{J^sH^n_{\mathfrak{m},J}(R)}\cong D(\Gamma_{J}(K_R)),$$ where $K_R$ is the canonical module of $R$.
\end{cor}
\begin{proof}We can consider the following isomorphims
$$
\begin{array}{lll}
 H^n_{\mathfrak{m},J}(R)/J^sH^n_{\mathfrak{m},J}(R)&\cong &  H^n_{\mathfrak{m},J}(R)\otimes_R R/J^n\\
&\cong &  H^n_{\mathfrak{m},J}(R/J^s) \ \ \ \ \ \ {\rm ( by \   \cite[Lemma 4.8]{art1})} \\
&\cong & H^n_{\mathfrak{m}}(R/J^s) \ \ \ \ \ \ {\rm ( by  \  \cite[Corollary 2.5]{art1})} \\
&\cong &  D({\rm Ext}^0(R/J^s, D(H^n_{\mathfrak{m}}(R))))\ \ \ \ \ \ {\rm ( by  \ Theorem  \ref{glduality})}
\end{array}
$$

Applying the inverse limit we obtain that
$$\varprojlim \frac{H^n_{\mathfrak{m},J}(R)}{J^sH^n_{\mathfrak{m},J}(R)}\cong \varprojlim D({\rm Ext}^0(R/J^s, D(H^n_{\mathfrak{m}}(R)))).$$

Since $$ \varprojlim D({\rm Ext}^0(R/J^s, D(H^n_{\mathfrak{m}}(R))))\cong D(\Gamma_J(D(H^n_{\mathfrak{m}}(R)))),$$ we finish the proof combining the isomorphism.

\end{proof}

The reader can see in \cite[Remark 5.5]{art1} that the previous isomorphismis is however not true.

\section{Endomorphism rings for a pair of ideals}
\hspace{0.5cm} In this section we start the investigation of ${\rm End}_R(H^t_{I,J}(M))$. We will give a alternative characterization of the smallest integer such that the local cohomology defined by a pair of ideals is non-zero. Also we show the relation between ${\rm End}_R(H^t_{I,J}(M))$ and ${\rm End}_R(D(H^t_{I,J}(M)))$ and several sufficient conditions for the homomorphism $$R\rightarrow  {\rm Hom}_R(D(H^t_{I,J}(M), D(H^t_{I,J}(M)))$$ is an isomorphism.
Firstly some preliminaries results are useful.

\begin{lem}\label{lemma2.2waqas} Let $(R,\mathfrak{m})$ be a local ring and $M,N$ two $R$-modules. There are the following isomorphisms for all $i\in \mathbb{Z}$:
\begin{enumerate}
\item[(a)] ${\rm Ext}_R^i(N,D(M))\cong D({\rm Tor}_i^R(N,D(M))).$

\item[(b)] If $N$ is finitely generated then
$$D({\rm Ext}_R^i(N,M))\cong {\rm Tor}_i ^R(N,M).$$
\end{enumerate}
\end{lem}
\begin{proof}
See \cite[Theorem 3.4.14]{Strooker}.
\end{proof}

\begin{lem}\label{lemma2.2}Let $(R,\mathfrak{m})$  be a local ring and $I,J$ be  ideals of $R$.  Consider $M$, $N$ two $R$-modules. Suppose that ${\rm Supp}_R M\subseteq W(I,J)$. Then there are the natural isomorphims:

\begin{enumerate}

\item [(a)] ${\rm Hom}_R(M, \Gamma_{I,J}(N))\cong {\rm Hom}_R(M,N)$;

\item [(b)] $M \otimes {\rm Hom}_R(\Gamma_{I,J}(N), E)\cong M\otimes {\rm Hom}_R(N,E)$.
\end{enumerate}
\end{lem}
\begin{proof}For the case that $M$ is finitely generated  or not finitely generated $R$-module, the statement $(a)$ is true.
For the statement $(b)$ consider initially that $M$ is finitely generated $R$-module.

By \cite[Lemma 10.2.16]{grot} we have the following isomorphism
$$ M \otimes_R {\rm Hom}_R(\Gamma_{I,J}(N), E)\cong {\rm Hom}_R({\rm Hom}_R(M,\Gamma_{I,J}(N)),E).$$
Thus, by item $(a)$ we conclude the proof.

Now, suppose that $M$ is not finitely generated $R$-module. Then $M\cong \displaystyle\lim_{ ^{\longrightarrow} }M_\alpha$ with $M_\alpha$ finitely generated $R$-modules.
The proof of $(b)$ is finished using  basic properties involving commutativity of direct limit with ${\rm Hom}$ and tensor product.
\end{proof}

The next result generalizes \cite[Theorem 2.3]{ScheMatlis} and \cite[Proposition 2.1]{waqas}.
\begin{thm}\label{theo2.3}Let $(R,\mathfrak{m})$  be a local ring and $I,J$ be  ideals of $R$.  Consider ${\rm depth}(I,J,N)= t$ and $M$ be a finitely generated  $R$-module such that ${\rm Supp}_R M\subseteq W(I,J)$ and . Then,

\begin{enumerate}
\item [(a)] ${\rm Ext}^t_R(M, N)\cong {\rm Hom}_R(M, H^t_{I,J}(N))$ and ${\rm Ext}^i_R(M, N)=0$ for all $i<t$  ;

\item [(b)] ${\rm Tor}_t^R(M, D(N) )\cong M \otimes_ R D(H^t_{I,J}(N))$ and ${\rm Tor}_i^R(M, D(N))=0$ for all $i<t$ .
\end{enumerate}
\end{thm}
\begin{proof} Consider $E_R^\bullet$ denote a minimal injective resolution of $R$-module $N$. We can describe each $E^i_R$ as a direct sum of indecomposable injective modules
$$E^i_R= \bigoplus_{\mathfrak{p}\in {\rm Spec}(R)}E_R(R/\mathfrak{p})^{\mu_i(\mathfrak{p},N)},$$
where $E_R(R/\mathfrak{p})$ denotes the injective hull of $R/\mathfrak{p}$ and $\mu_i(\mathfrak{p},N)$ is the $i$-th Bass number of $N$ with respect to $\mathfrak{p}$.

Since ${\rm depth}(I,J,N) =  {\rm inf}\{{\rm depth}(N_\mathfrak{p}) \mid \mathfrak{p}\in W(I,J)\}$ and ${\rm depth}(N_\mathfrak{p})= {\rm inf}\{ i\mid \mu_i(\mathfrak{p},N)\neq 0\}$, follows that $\mu_i(\mathfrak{p},N)=0$ for all $i<t$ and $\mathfrak{p}\in W(I,J)$. So, we can conclude that ${\rm Hom}_R(M,E^i_R)=0$ for all $i<t$. Since  ${\rm Supp}_R M\subseteq W(I,J)$, by Lemma \ref{lemma2.2} follows  the isomorphism of complexes
$${\rm Hom}_R(M,E^\bullet_R)\cong {\rm Hom}_R(M,\Gamma_{I,J}(E^\bullet_R)).$$

By the exact sequence $0\rightarrow H^t_{I,J}(N)\rightarrow \Gamma_{I,J}(E_R^\bullet)^t\rightarrow \Gamma_{I,J}(E_R^\bullet)^{t+1}$ and the previous isomorphism of complex we obtain a commutative diagram\\

$${\small
\begin{array}{lcr}
 \begin{array}{c}
 \\
 \\
 \\
 \\
 \end{array}&\begin{array}{ccccc}
 0&\!\!{\rightarrow} & {\rm Hom}_R(M, H^t_{I,J}(N)) \rightarrow & {\rm Hom}_R(M,\Gamma_{I,J}(E_R^\bullet))^t \rightarrow& {\rm Hom}_R(M,\Gamma_{I,J}(E_R^\bullet))^{t+1}  \\
   &  &  \downarrow  &  \downarrow & \downarrow\\
     0 & {\rightarrow} & {\rm Ext}_R^t(M, N) \rightarrow &
     {\rm Hom}_R(M,E_R^\bullet)^t \rightarrow& {\rm Hom}_R(M,E_R^\bullet)^{t+1}
   \end{array}
   & \begin{array}{c}
 \\
 \\
 \\
 \\
 \end{array}
\end{array}}
$$

\noindent with exact rows. We can see that, by Lemma \ref{lemma2.2}, that the last two vertical arrows are isomorphisms and so, the first vertical arrow is an isomorphism too. This complete the proof of $(a)$.

For the statement $(b)$ first, appling the Matlis duality functor $D(-)$ to the above exact sequence we obtain
$$D(\Gamma_{I,J}(E_R^{t+1}))\rightarrow D(\Gamma_{I,J}(E_R^t))\rightarrow D(H^t_{I,J}(N))\rightarrow 0.$$

Note that, since $D(\Gamma_{I,J}(E_R^\bullet))$ is a complex of flat $R$-modules, this is the start of a flat resolution of $D(H^i_{I,J}(N))$. Furthermore, $D(E_R^\bullet)\stackrel{\cong}\rightarrow E$
is a flat resolution of $E$. Using the fact that $M\cong \varinjlim M_\alpha$ with $M_\alpha$ finitely generated $R$-modules, by Lemma \ref{lemma2.2} we obtain the isomorphisms of complexes
$$M\otimes_R D(E_R^\bullet)\cong M\otimes_RD(\Gamma_{I,J}(E_R^\bullet))\cong \varinjlim{\rm Hom}_R({\rm Hom}_R(M_{\alpha}, \Gamma_{I,J}(E_R^\bullet)),E). $$

Since ${\rm Hom}_R(M,E^i_R)=0$ for all $i<t$ follows that $H_i(M\otimes_R D(E_R^\bullet)) =0$ for all $i<t.$ This shows that ${\rm Tor}_i^R(M, D(N))=0$ for all $i<t$.

\noindent Therefore, by commutative diagram with exact rows

$$
\begin{array}{ccccccc}
 M\otimes_R D(E_R^{t+1})& \rightarrow & M\otimes_R D(E_R^{t})& \rightarrow & {\rm Tor}_t^R(M, D(N))\rightarrow 0 \\
\downarrow &  & \downarrow &  &  \downarrow  \\
 M\otimes_R D(\Gamma_{I,J}(E_R^{t+1}))& \rightarrow & M\otimes_R D(\Gamma_{I,J}(E_R^{t})) &\rightarrow & M\otimes_R D(H^t_{I,J}(N)) \rightarrow  0
\end{array}
$$

\noindent finishes the proof of $(b)$  because two vertical arrows at the left are isomorphisms and so,  ${\rm Tor}_t^R(M, D(N))\cong M\otimes_R D(H^t_{I,J}(N))$.
\end{proof}

As a consequence of the previous result we show a characterization of depth of $M$ defined by a pair of ideals.

\begin{cor} Let $I,J$  be ideals of local ring $R$ and  $N$ be a finitely generated  $R$-module.
$${\rm depth}(I,J,N)= {\rm inf}\{ i \in \mathbb{Z} \mid {\rm Tor}_i^R(R/\mathfrak{a},D(N))\neq 0\  {\rm and}  \ \mathfrak{a}\in \widetilde{W}(I,J) \}.$$
\end{cor}
\begin{proof} First note that, by Lemma \ref{auxiliar},
$${\rm depth}(I,J,N)= {\rm inf}\{ {\rm depth}(\mathfrak{a},N) \mid \mathfrak{a}\in \widetilde{W}(I,J)\}.$$

Then, apply version of Theorem \ref{theo2.3} for the ideal $\mathfrak{a}$ or \cite[Corollary 2.3]{waqas},  we have that
$${\rm depth}(\mathfrak{a},N)= {\rm inf}\{ i \in \mathbb{Z} \mid {\rm Tor}_i^R(R/\mathfrak{a},D(N))\neq 0\}.$$ This finishes the proof of the statement.

\end{proof}
The next result show the close relation between ${\rm End}_R(H^t_{I,J}(M))$ and ${\rm End}_R(D(H^t_{I,J}(M)))$. This  result and the next corollary are a generalization of \cite[Theorem 1.1]{ScheMatlis} and \cite[Lemma 3.2]{waqas}.

\begin{thm}\label{lem3.2waqas} Let $(R,\mathfrak{m})$ be a complete local ring of dimension $n$ and $I,J$ be two ideals of $R$. For an finitely generated $R$-module $M$ there is a natural isomorphism
$${\rm Hom}_R(H^t_{I,J}(M), H^t_{I,J}(M))\cong {\rm Hom}_R(D(H^t_{I,J}(M), D(H^t_{I,J}(M))),$$
where ${\rm depth}(I,J,M)=t$.
\end{thm}
\begin{proof} First note that, by Lemma \ref{lemma2.2waqas} follows the isomorphism
$${\rm Hom}_R(D(H^t_{I,J}(M), D(H^t_{I,J}(M)))\cong D(H^t_{I,J}(M)\otimes_R D(H^t_{I,J}(M)) ).$$
Since, by \cite[Corollary 1.13,(5) and Proposition 1.7]{art1}, $H^t_{I,J}(M)$ is $(I,J)$-torsion $R$-module we can apply the Theorem \ref{theo2.3} $(b)$ and we obtain that

$$
\begin{array}{lll}
 D(H^t_{I,J}(M)\otimes_R D(H^t_{I,J}(M))) &\cong &  D({\rm Tor}_t^R(H^t_{I,J}(M),D(M)))\\
&\cong & {\rm Ext}^t_R(H^t_{I,J}(M),D(D(M))).
\end{array}
$$

Since $R$ is complete local ring, by Matlis Duality Theorem \cite[Theorem 10.2.12]{grot}, follow that $D(D(M))\cong M$.
Applying again Theorem \ref{theo2.3} $(a)$ we obtain that
$${\rm Ext}^t_R(H^t_{I,J}(M),D(D(M)))\cong {\rm Hom}_R(H^t_{I,J}(M), H^t_{I,J}(M))$$ and this complete the proof.
\end{proof}

\begin{cor}\label{lemma3.2i}Let the same hypothesis of the Theorem \ref{lem3.2waqas}. Then the natural homomorphism
$$R\rightarrow {\rm Hom}_R(H^t_{I,J}(M), H^t_{I,J}(M))$$ is an isomorphism if, and only if, the natural homomorphism
$$R\rightarrow  {\rm Hom}_R(D(H^t_{I,J}(M), D(H^t_{I,J}(M)))$$ is an isomorphism if, and only if, the natural homomorphism
$$H^t_{I,J}(M)\otimes_R D(H^t_{I,J}(M))\rightarrow E$$ is an isomorphism.
\end{cor}
\begin{proof} By previous theorem follows that the natural homomorphism
$$R\rightarrow {\rm Hom}_R(H^t_{I,J}(M), H^t_{I,J}(M))$$ is an isomorphism if, and only if, the natural homomorphism
$$R\rightarrow  {\rm Hom}_R(D(H^t_{I,J}(M), D(H^t_{I,J}(M)))$$ is an isomorphism.

Therefore if, the natural homomorphism
$$R\rightarrow  {\rm Hom}_R(D(H^t_{I,J}(M), D(H^t_{I,J}(M)))$$ is an isomorphism then, by Matlis duality, the natural homomorphism
$$D(R)\rightarrow  D({\rm Hom}_R(D(H^t_{I,J}(M), D(H^t_{I,J}(M))))$$ is an isomorphism.
By Lemma \ref{lemma2.2waqas}(a) we have that
$$D({\rm Hom}_R(D(H^t_{I,J}(M), D(H^t_{I,J}(M))))\cong H^t_{I,J}(M)\otimes_R D(H^t_{I,J}(M)),$$ and this finishes the proof.
\end{proof}

\section{The Truncation Complex for a pair of ideals}

\hspace{0.5cm}
Recall that the truncation complex  was introduced in \cite[Section 2]{HellusSche} when $(R, \mathfrak{m},k)$  is a $d$-dimensional local Gorenstein ring, and in a different approach in \cite{ScheMac} and \cite{waqas}. We generalize this concept using local cohomology defined by a pair of ideals. This definition will be key to the most important result of this section (Theorem \ref{4.3}).

Let $(R,\mathfrak{m})$ be a $d$-dimensional local ring and $M\neq 0$ a finitely generated $R$-module. Consider $I,J$ ideals of $R$ with ${\rm depth}(I,J,M)=t$ and ${\rm dim}M/JM=n$.

Let $E_R^\bullet(M)$ denote a minimal injective resolution of $R$-module $M$. It's well knows that we can describe each $E^i_R(M)$ as a direct sum of indecomposable injective modules
$$E^i_R(M)= \bigoplus_{\mathfrak{p}\in {\rm Spec}(R)}E_R(R/\mathfrak{p})^{\mu_i(\mathfrak{p},M)},$$
where $E_R(R/\mathfrak{p})$ denotes the injective hull of $R/\mathfrak{p}$ and $\mu_i(\mathfrak{p},N)$ is the $i$-th Bass number of $M$ with respect to $\mathfrak{p}$.
By,  \cite[Proposition 1.11]{art1} follows that
$$\Gamma_{I,J}(E^i_R(M))= \bigoplus_{\mathfrak{p}\in W(I,J)}E_R(R/\mathfrak{p})^{\mu_i(\mathfrak{p},M)},$$

Since ${\rm depth}(I,J,M) =  {\rm inf}\{{\rm depth}(M_\mathfrak{p}) \mid \mathfrak{p}\in W(I,J)\}$ and ${\rm depth}(M_\mathfrak{p})= {\rm inf}\{ i\mid \mu_i(\mathfrak{p},M)\neq 0\}$, follows that $\mu_i(\mathfrak{p},M)=0$ for all $i<t$ and $\mathfrak{p}\in W(I,J)$. Then, for all $i<t$ we have that $\Gamma_{I,J}(E^i_R(M)) =0$.

Therefore $H^t_{I,J}(M)$ is isomorphic to the kernel of $$\Gamma_{I,J}(E^t_R(M))\rightarrow \Gamma_{I,J}(E^{t+1}_R(M))$$ and so, there is an embedding of complexes of $R$-modules
$$H^t_{I,J}(M)[-t]\rightarrow \Gamma_{I,J}(E^\bullet_R(M)).$$

\begin{defn} We call the cokernel   of the above embedding, denoted by $C^\bullet_M(I,J)$, by \emph{ truncation complex with respect to pair of ideals (I,J)}. Thus, we can consider the short exact sequence of complexes of $R$-moules
$$0\rightarrow H^t_{I,J}(M)[-t]\rightarrow \Gamma_{I,J}(E^\bullet_R(M))\rightarrow C^\bullet_M(I,J)\rightarrow 0.$$
\end{defn}

Note that by this definition, $H^i(C^\bullet_M(I,J))=0$ for all $i<t$. Futhermore, if $i>n$ ( by \cite[Theorem 4.3]{art1} with $J\neq R$)   and $i> {\rm dim}R/J$ we have too that $H^i(C^\bullet_M(I,J))=0$.

\begin{lem}\label{lemma3.3}  Let $(R,\mathfrak{m})$ be a complete local ring of dimension $n$ and $I,J$ be two ideals of $R$. Let an finitely generated $R$-module $M$ such that ${\rm depth}(I,J,M)=t$ and $H^i_{I,J}(M)=0$ for all $i\neq t$. Then for all integer $i\neq c$ fixed:

\begin{itemize}
\item[(a)] Follows the isomorphisms:
\begin{itemize}
\item[(i)] ${\rm Ext}_R^{i-t}(H^t_{I,J}(M), H^t_{I,J}(M))\cong {\rm Ext}_R^i(H^t_{I,J}(M), M)$.
\item[(ii)] ${\rm Tor}_{i-t}^R(H^t_{I,J}(M), D(H^t_{I,J}(M)))\cong {\rm Tor}_i^R(H^t_{I,J}(M), D(M))$.
\end{itemize}
\item[(b)] The following conditions are equivalent:
\begin{itemize}
\item[(i)] ${\rm Ext}_R^{i-t}(H^t_{I,J}(M), H^t_{I,J}(M))=0$;
\item[(ii)] ${\rm Ext}_R^{i-t}(D(H^t_{I,J}(M)), D(H^t_{I,J}(M)))=0$;
\item[(iii)] ${\rm Tor}_{i-t}^R(H^t_{I,J}(M), D(H^t_{I,J}(M)))=0$.
\end{itemize}
\end{itemize}
\end{lem}
\begin{proof}Let $E_R^\bullet(M)$ denote a minimal injective resolution of $R$-module $M$. Note that the complex $\Gamma_{I,J}(E_R^\bullet(M))$ is a minimal injective resolution of $H^t_{I,J}(M)[-t]$ because $H^i_{I,J}(M)=0$ for all $i\neq t$. By \cite[Corollary 1.13 and Proposition 1.7]{art1} we have that ${\rm Supp}_R(H^t_{I,J}(M))\subseteq W(I,J)$. So, by \ref{lemma2.2} we have the isomorphism
$${\rm Ext}_R^{i-t}(H^t_{I,J}(M), H^t_{I,J}(M))\cong {\rm Ext}_R^{i}(H^t_{I,J}(M), M),$$ for all integer $i$. Thus, we obtain the first statement of $(a)$.
For the second clain, first note that the complexes $ D(E_R^\bullet(M))$ and $D(\Gamma_{I,J}(E_R^\bullet(M)))$  are flat resolutions of $D(M)$ and $D(H^t_{I,J}(M)[t])$ respectively.

Since ${\rm Supp}_R(H^t_{I,J}(M))\subseteq W(I,J)$ , by Lemma \ref{lemma2.2} follows the isomorphism
$$H^t_{I,J}(M)\otimes_R D(E_R^\bullet(M)) \cong H^t_{I,J}(M)\otimes_R D(\Gamma_{I,J}(E_R^\bullet(M)))$$
that induces, for all integer $i$, the isomorphisms in homologies
$${\rm Tor}_{i-t}^R(H^t_{I,J}(M), D(H^t_{I,J}(M)))\cong {\rm Tor}_i^R(H^t_{I,J}(M), D(M)).$$
This finishes the proof of $(a)$.

For $(b)$, by Lemma \ref{lemma2.2waqas} there are  isomorphisms
$${\rm Ext}_R^{i-t}(D(H^t_{I,J}(M)), D(H^t_{I,J}(M)))\cong D({\rm Tor}^R_{i-t}(D(H^t_{I,J}(M)), H^t_{I,J}(M))) \ \ \ (1) $$ and
$$D({\rm Tor}_i^R(H^t_{I,J}(M), D(M)))\cong {\rm Ext}_R^i(H^t_{I,J}(M), M) \ \ \ \ \ (2),$$
 for all integer $i$. Therefore, using $(a)$ and isomorphism $(1)$ and $(2)$ we obtain the  vanishing of claims of $(b)$.
\end{proof}

We are now ready to prove the most important result of this section.

\begin{thm}\label{4.3}Let $(R,\mathfrak{m})$ be a  $d$-dimensional complete Cohen Macaulay local ring. Consider $I,J$ ideals of $R$ such that ${\rm depth}(I,J,R)=t$ and $H^i_{I,J}(R)=0$ for all $i\neq t$.
\begin{itemize}
\item[(a)] The natural homomorphism
$$R\rightarrow {\rm Hom}_R(H^t_{I,J}(K_R), H^t_{I,J}(K_R))$$ is an isomorphism and for all $i\neq t$
$${\rm Ext}_R^{i-t}(H^t_{I,J}(K_R), H^t_{I,J}(K_R))=0.$$

\item[(b)] The natural homomorphism
$$R\rightarrow {\rm Hom}_R(D(H^t_{I,J}(K_R)), D(H^t_{I,J}(K_R)))$$ is an isomorphism and for all $i\neq t$
$${\rm Ext}_R^{i-t}(D(H^t_{I,J}(K_R)), D(H^t_{I,J}(K_R)))=0.$$

\item[(c)] The natural homomorphism
$$H^t_{I,J}(K_R)\otimes_R D(H^t_{I,J}(K_R))\rightarrow E$$ is an isomorphism and for all $i\neq c$
$${\rm Tor}_{i-t}^R(H^t_{I,J}(K_R), D(H^t_{I,J}(K_R)))=0.$$
\end{itemize}
\end{thm}
\begin{proof} Note that it's sufficient to show the claim $(a)$ by Lemma \ref{lemma3.3}, Theorem \ref{lem3.2waqas} and Corollary \ref{lemma3.2i}.
By Lemma \ref{auxiliar} we have that
$${\rm depth}(I,J,N)= {\rm inf}\{ {\rm depth}(\mathfrak{a},N) \mid \mathfrak{a}\in \widetilde{W}(I,J)\}.$$
Since the canonical module $K_R$ of $R$ exists and
$${\rm depth}(\mathfrak{a},K_R)= {\rm dim}_R(K_R)- {\rm dim}_R(K_R/\mathfrak{a}K_R),$$ for all $\mathfrak{a}\in \widetilde{W}(I,J)$ follows that ${\rm depth}(\mathfrak{a},K_R)= {\rm depth}(\mathfrak{a},R)$ and so, $t= {\rm depth}(I,J,K_R).$

Let $E^\bullet_R(K_R)$ be a minimal injective resolution of $K_R$. To apply de functor ${\rm Hom}_R(-,E^\bullet_R(K_R))$ to the short exact sequence of truncation complex of $K_R$ with respect to ideals $(I,J)$ we obtain
$$0\rightarrow {\rm Hom}_R(C^\bullet_{K_R}(I,J),E^\bullet_R(K_R)) \rightarrow {\rm Hom}_R(\Gamma_{I,J}(E^\bullet_R(K_R)),E^\bullet_R(K_R))$$
$$\rightarrow {\rm Hom}_R(H^t_{I,J}(K_R),E^\bullet_R(K_R))[t] \rightarrow 0.$$
By \cite[Theorem 3.2]{art1} the middle component of this sequence is isomorphic to
$$
\begin{array}{lll}
\displaystyle \varprojlim_{\mathfrak{a}\in \widetilde{W}(I,J)}{\rm Hom}_R(\Gamma_{\mathfrak{a}}(E^\bullet_R(K_R),E^\bullet_R(K_R)) &\cong &  \displaystyle \varprojlim_{\mathfrak{a}\in \widetilde{W}(I,J)}\displaystyle \varprojlim{\rm Hom}_R(\Gamma_{\mathfrak{a}}(E^\bullet_R(K_R)),E^\bullet_R(K_R)) \\
& \cong & \displaystyle\varprojlim_{\mathfrak{a}\in \widetilde{W}(I,J)}\displaystyle \varprojlim (R/\mathfrak{a}^r \otimes_R {\rm Hom}(E^\bullet_R(K_R),E^\bullet_R(K_R)))   \\
\end{array}
$$

For the previous isomorphism  note that $R/\mathfrak{a}^r$ is finitely generated $R$-module for all $r\geq 1$. Consider $Y:= {\rm Hom}(E^\bullet_R(K_R),E^\bullet_R(K_R))$. Note that there is a quasi-isomorphism between $X$ and ${\rm Hom}(K_R,E^\bullet_R(K_R)).$

Follows the definition ${\rm Hom}$ of complexes that
$$Y^j\cong \prod_{i\in \mathbb{Z}} {\rm Hom}(E^i_R(K_R),E^{i+j}_R(K_R)),$$
and so for all $j\in\mathbb{Z}$, $X^j$ is a flat $R$-module. Since $R$ is a $d$-dimensional Cohen Macaulay, $H^i_{\mathfrak{m}}(R)=0$ for all $i\neq d$. Applying Theorem \ref{glduality} for $I=\mathfrak{m}$ and $J=0$ follows that $H^j(Y)\cong {\rm Ext}_R^j(K_R,K_R)=0$ for all $j\neq 0$ and $H^0(Y)= {\rm Hom}_R(K_R,K_R)\cong R$. Also, if $\mathfrak{p}\in {\rm Spec}R$ we can see that
$$E^i_R(K_R)\cong \bigoplus_{{\rm height} \mathfrak{p}=i}E_R(R/\mathfrak{p}).$$

Since $R$ is Cohen Macaulay, $K_R$ have finite injective dimension. So, $Y^j=0$ for all $k>0$. By this we can conclude that $Y$ turns into a flat resolution of $R$. So, we can see that the cohomologies of the complex
$$\displaystyle \varprojlim_{\mathfrak{a}\in \widetilde{W}(I,J)}\displaystyle \varprojlim(R/\mathfrak{a}^r \otimes_R {\rm Hom}(E^\bullet_R(K_R),E^\bullet_R(K_R)))$$ are zero for all $i\neq 0$ and, since $R$ is complete, for $i=0$ is $R$. Furthermore $H^i_{I,J}(K_R)=0$ for all $i\neq c$. Thus the complex ${\rm Hom}_R(C^\bullet_{K_R}(I,J),E^\bullet_R(K_R))$ is an exact complex.

Apply the long exact sequence of cohomologies of the previous exact sequence, follows the isomorphism
$$R\rightarrow {\rm Ext}_R^t(H^t_{I,J}(K_R),K_R)$$ and, for all $i\neq t$
$${\rm Ext}_R^i(H^t_{I,J}(K_R),K_R)=0.$$

By Theorem \ref{theo2.3} and  Lemma \ref{lemma3.3} $(a)$, finishes the proof of statement $(a)$.
\end{proof}


\begin{thebibliography}{PTW02}

\bibitem{art8} M. Aghapournahr, KH. Ahmadi-Amoli and M.Y. Sadegui, \emph{The concepts of depth of a pair of ideals $(I,J)$ on modules and $(I,J)$-Cohen-Macaulay modules}, arXiv:1301.1015, (2013).



\bibitem{serresub} M. Aghapournahr and L. Melkersson, \emph{Local cohomology and Serre subcategories}, J. Algebra, \textbf{320}, (2008), 1275-1287.

\bibitem{amoli} Kh. Ahmadi-Amoli and M. Y. Sadeghi, \emph{On the Local Cohomology Modules Defined by a Pair of Ideals and Serre Subcategory}, J. of Math. Extension, \textbf{7}, (2013), 47-62.

%bibitem{finiteness} M. Asgharzadeh and K. Divaani- Aazr, \emph{Finiteness Properties of Formal Local Cohomology modules and Cohen- Macaulayness}, Com. Algebra, \textbf{39-3} (2011), 1082-1103.

%\bibitem{Bourbaki} N.Bourbaki, \emph{Alg\'ebre commutative}, Hermann, Paris, 1961-1965.

\bibitem{grot} M.P Brodmann and R. Y. Sharp, \emph{Local cohomology- an algebraic introduction with geometric applications}, Cambridge University Press,(1998).

\bibitem{artop}L. Chu, \emph{Top local cohomology modules with respect to a pair of ideals} Proc. Amer. Math. Soc., \textbf{139-3}
  (2010), 777-782.

\bibitem{art3}
 L. Chu, and Q. Wang, \emph{Some results on local cohomology modules defined by a pair of ideals} J.Math. Kyoto Univ, \textbf{49-1}
  (2009), 193-200.

%bibitem{8} K. Divaani- Aazr, R. Sazeedeh, M. Tousi, \emph{On the vanishing of generalized local cohomology modules}, Algebra Colloq., \textbf{12} (2005), 213-218.


%bibitem{6} K. Divaani- Aazr and P. Schenzel , \emph{Ideal Topology, local cohomology and connectedness}, Math. Proc. Cambridge philos. Soc., \textbf{131} (2001), 211-226.

%bibitem{faltings} G. Faltings, \emph{Algebraization of some formal vector bundles}, Ann. of Math. 110 (1979) 501-514.

%bibitem{majid} M. Eghbali, \emph{On Formal local cohomology, colocalization and endomorphism ring of top local cohomology modules}, (2011)	Thesis, Universitat Halle-Wittenberg.

\bibitem{Hart}A. Grothendieck,  \emph{Local Cohomology}, Notes by R. Hartshorne, Lecture Notes in Math., vol 20, Springer, 1966.

%\bibitem{yan} Gu. Y, \emph{ The Artinianness of Formal Local Cohomology Modules} *********

%\bibitem{herzog} J. Herzog, \emph{Komplexe }, Auflösungen und Dualität in der lokalen Algebra, Habilitationsschrift, Universität Regensburg (1970).

\bibitem{HellusSche} M. Hellus and P. Schenzel, \emph{On cohomologically complete intersections}, J. Algebra, \textbf{320}, (2008), 3733-3748.

\bibitem{HellusStuck} M. Hellus and J. Stuckrad, \emph{On Endomorphism Rings of Local Cohomology Modules}, Proc. Am. Math. Soc. , \textbf{136}, (2008), 2333-2341.


\bibitem{huneke} C. Huneke, \emph{Problems on local cohomology}, Free resolutions in commutative algebra and algebraic geometry, Res. Notes Math., \textbf{2} (1992), 93-108.

\bibitem{hochhuneke}C. Huneke and M. Hochster, \emph{Indecomposable canonical modules and connectedness}, Comm. Algebra: syzygies, multiplicities, and birational algebra, Providence, RI: American Mathematical Society. Contemporary Mathematics \textbf{159}, (1994), 197-208.

\bibitem{24h} S. B. Iyengar, G.J. Leuschke, A. Leykin, C. Miller, E. Miller, A.K. Singh, U. Walther,\emph{ Twenty-Four
Hours of Local Cohomology}, Graduate Studies in Mathematics 87, Amer. Math. Soc. 2007.

\bibitem{paper} V.H. Jorge Perez T.H. Freitas, \emph{On formal local cohomology defined by a pair of ideals}, (in preparation).

\bibitem{paper2} V.H. Jorge Perez T.H. Freitas, \emph{Artinianness and Finiteness Properties of Formal Local Cohomology Modules Defined by a Pair of Ideais}, (in preparation).



\bibitem{Kazem} K. Khashyarmanesh, \emph{On the Endomorphism Rings of Local Cohomology Modules}, Canad. Math. Bull. \textbf{53} (2010), no. 4, 667-673.


\bibitem{13} M. Mafi, \emph{Some results on the local cohomology modules}, Arch Math (Basel), \textbf{87} (2006), 211-216.


\bibitem{waqas} W. Mahmood, \emph{On Endomorphism Ring of Local Cohomology Modules}, (2013), arXiv:1308.2584.

%bibitem{131} M. Mafi, \emph{Results on formal local cohomology modules}, Bull. Malays. Math. Sci. Soc, (2) \textbf{36} (2013),nº 1, 173-177.

%bibitem{14} T. Marley and J.C Vassilev, \emph{Local cohomology modules with infinite dimension socles}, Proc. Amer. Math. Soc , \textbf{132} (2004), 3485-3490.

%bibitem{16} L. Melkerson, \emph{Some applications of a criterion of artiniannes of a module}, J. Pure and Applied Algebra, \textbf{101} (1995), 291-303.

%bibitem{melschen} L. Melkerson and P. Schenzel, \emph{The co-localization of an Artinian module}, Proc. Endinburgh Math. Soc., \textbf{38}(1995), 121-131.

%\bibitem{art2} T.T Nam  and  N. M. Tri,
%\emph{Some properties of generalized local cohomology modules with respect %to a pair of ideals},(2013),	arXiv:1011.4141 .

\bibitem{payrovi} Sh. Payrovi and M. Lofti Parsa, \emph{Finiteness of Local Cohomology Modules Defined by a pair of ideals}, Comm. Algebra \textbf{41} (2013), no. 2, 627-637

\bibitem{payrovi2} Sh. Payrovi and M. Lofti Parsa, \emph{Artinianness of local cohomology Modules defined by a pair of ideals}, B. Malays. Math. Sci. So.,  \textbf{53} (2010), 577-586.

%bibitem{18} C. Peskine and L. Szpiro , \emph{Dimension projective finie et cohomologie locale}, Publ. Math. I.H.E.S, \textbf{42} (1972), 47-119.

\bibitem{brod}J. Rotman, \emph{An introduction to homological algebra}, Second Edition, Academic Press, Orlando, FL, 1979.

\bibitem{ScheMac} P. Schenzel,\emph{ On Macaulayfications and Cohen-Macaulay canonical modules}, J. Algebra \textbf{275} (2004), 751-770.

\bibitem{ScheEnd} P. Schenzel,\emph{ On Endomorphism Rings and Dimensions of Local Cohomology Modules}, Proc. Amer. Math. Soc \textbf{137} (2009), 1315-1322.

\bibitem{ScheMatlis} P. Schenzel,\emph{ Matlis Dual of Local Cohomology Modules and their Endomorphism Rings}, Arch. Math., \textbf{95} (2010), 115-123.

\bibitem{ScheStruct} P. Schenzel,\emph{ On the Structure of the Endomorphism Rings of a Certain Local Cohomology Module}, J. Algebra, \textbf{344} (2011), 229-245.

\bibitem{notas0}
P. Schenzel,  \emph{On the use of local cohomology in Algebra and Geometry}, Notes.

\bibitem{Strooker} J. R. Strooker Homological Questions in Local Algebra, London mathematical society, Lectures Notes Sries. Cambribge University Press (1990).


\bibitem{art1}
T. Takahashi, Y. Yoshino, and  T. Yoshizawa, \emph{Local cohomology based on a nonclosed support defined by a pair of ideals},J. Pure Appl. Algebra, \textbf{213},(2009), 582-600.

\bibitem{nonart}A. Tehranian and A.P.E Talemi, \emph{Non-Artinian Local Cohomology with Respect to a Pair of Ideals}, A. Colloquium, \textbf{20}: 4 (2013), 637-642.

\bibitem{filterdepth}A. Tehranian and A.P.E Talemi, \emph{Filter Depth and cofiniteness of Local Cohomology Modules defined by  a Pair of Ideals}, A. Colloquium, \textbf{21}: 4 (2014),597-604.

%bibitem{yassemi}S. Yassemi, \emph{Coassociated primes}, Comm. Algebra, \textbf{23} (4), (1995), 1473-1498.

\bibitem{Weibel}
C.A, Weibel , \emph{An introduction to homological algebra}, Cambridge University Press, 1994

\end{thebibliography}
\end{document}